\documentclass[11pt]{article}
\usepackage{amsthm,amsfonts, amsbsy, amssymb,amsmath,graphicx}
\usepackage{graphics}
\usepackage[english]{babel}
\usepackage{graphicx}
\usepackage{lineno}
\usepackage{enumerate}
\usepackage{tikz}
\usepackage{hyperref}
\usepackage{color}
\usepackage{tikz}

\hypersetup{colorlinks=true}

\hypersetup{colorlinks=true, linkcolor=blue, citecolor=blue,urlcolor=blue}

\title{Revisiting $k$-tuple dominating sets with emphasis on small values of $k$}

\author{Babak Samadi$^a$\thanks{Corresponding author}, Nasrin Soltankhah$^a$ and Doost Ali Mojdeh$^b$\\[0.5cm]
$^{a}$Department of Mathematics, Faculty of Mathematical Sciences, Alzahra University,\\ Tehran, Iran\\
{\it b.samadi@alzahra.ac.ir}\\
{\it soltan@alzahra.ac.ir}\\[0.2cm]
$^b$Department of Mathematics, University of Mazandaran,\\ Babolsar, Iran\\
{\it damojdeh@umz.ac.ir}
}
\date{}

\newtheorem{theorem}{Theorem}[section]
\newtheorem{corollary}[theorem]{Corollary}

\newtheorem{proposition}[theorem]{Proposition}

\theoremstyle{definition}

\theoremstyle{remark}
\newtheorem{rem}[theorem]{Remark}

\begin{document}

\maketitle

\begin{abstract}
For any graph $G$ of order $n$ with degree sequence $d_{1}\geq\cdots\geq d_{n}$, we define the double Slater number $s\ell_{\times2}(G)$ as the smallest integer $t$ such that $t+d_{1}+\cdots+d_{t-e}\geq2n-p$ in which $e$ and $p$ are the number of end-vertices and penultimate vertices of $G$, respectively. We show that $\gamma_{\times2}(G)\geq s\ell_{\times2}(G)$, where $\gamma_{\times2}(G)$ is the well-known double domination number of a graph $G$ with no isolated vertices. We prove that the problem of deciding whether the equality holds for a given graph is NP-complete even when restricted to $4$-partite graphs. We also prove that the problem of computing $\gamma_{\times2}(G)$ in NP-hard even for comparability graphs of diameter two. Some results concerning these two parameters are given in this paper improving and generalizing some earlier results on double domination in graphs. We give an upper bound on the $k$-tuple domatic number of graphs with characterization of all graphs attaining the bound. Finally, we characterize the family of all full graphs, leading to a solution to an open problem given in a paper by Cockayne and Hedetniemi ($1977$).

\noindent \ \ \   
\end{abstract}
\vspace{1mm}


\section{Introduction and preliminaries}

Throughout this paper, we consider $G$ as a finite simple graph with vertex set $V(G)$ and edge set $E(G)$. We use \cite{West} as a reference for terminology and notation which are not explicitly defined here. The {\em open neighborhood} of a vertex $v$ is denoted by $N(v)$, and its {\em closed neighborhood} is $N[v]=N(v)\cup \{v\}$. We denote the {\em degree} of vertex $v$ by $\deg(v)$, and let $\deg_{S}(v)=|N(v)\cap S|$ in which $S\subseteq V(G)$. The {\em minimum} and {\em maximum degrees} of $G$ are denoted by $\delta(G)$ and $\Delta(G)$, respectively. An {\em end-vertex} is a vertex of degree one and a {\em penultimate vertex} is a vertex adjacent to an end-vertex (they are called {\em leaf} and {\em support vertex} in the case of trees). Given the subsets $A,B\subseteq V(G)$, by $[A,B]$ we mean the set of all edges with one end point in $A$ and the other in $B$. Finally, for a given set $S\subseteq V(G)$, by $G[S]$ we represent the subgraph induced by $S$ in $G$.

A {\em transitive orientation} of a graph $G$ is an orientation $D$ such that whenever $(x,y)$ and $(y,z)$ are arcs in $D$, also there exists an edge $xz$ in $G$ that is oriented from $x$ to $z$ in $D$. A graph $G$ is a {\em comparability graph} if it has a transitive orientation. It is well-known that the comparability graphs are perfect (see \cite{West} for example).

A set $S\subseteq V(G)$ is a {\em dominating set} if each vertex in $V(G)\setminus S$ has at least one neighbor in $S$. The {\em domination number} $\gamma(G)$ is the minimum cardinality of a dominating set in $G$. For more information about domination and its related parameters the reader can consult \cite{hhh} and \cite{hhs}.  

Slater \cite{s} showed that the domination number of a graph of order $n$ with non-increasing degree sequence $d_{1}\geq\cdots\geq d_{n}$ can be bound from below by the smallest integer $t$ such that $t$ added to the sum of the first $t$ terms of the above-mentioned sequence is at least $n$. This parameter was first called the {\em Slater number} and denoted by $s\ell(G)$ in \cite{dhh}. This parameter and its properties have been investigated in \cite{ghr} and \cite{gr}. 

A vertex subset $S$ of a graph $G$ with $\delta(G)\geq k-1$ is said a {\em $k$-tuple dominating set} if $|N[v]\cap S|\geq k$ for each vertex $v$ of $G$. The {\em $k$-tuple domination number} $\gamma_{\times k}(G)$ of the graph $G$ is the minimum cardinality of a $k$-tuple dominating set in $G$. A $k$-tuple dominating set in $G$ of the minimum cardinality is called a {\em $\gamma_{\times k}(G)$-set}. This parameter was introduced by Harary and Haynes in \cite{hh}. For the especial case $k=2$, it is common to write \textit{double dominating set} and \textit{double domination number} for the resulting set and graph parameter. Note that we emphasis on the small values of $k$ for various reasons given in this paper.

For a given graph $G$ of order $n$ with $e$ end-vertices, $p$ penultimate vertices and non-increasing degree sequence $d_{1}\geq\cdots\geq d_{n}$, we define the {\em double Slater number} $s\ell_{\times2}(G)$ as follows:
$$s\ell_{\times2}(G)=\mbox{min}\{t\mid t+d_{1}+\cdots+d_{t-e}\geq2n-p\}.$$

A vertex partition of a graph $G$ with $\delta(G)\geq k-1$ is said to be a $k$-tuple domatic partition of $G$ if each partite set is a $k$-tuple dominating set in $G$. The $k$-tuple domatic number $d_{\times k}(G)$ is the maximum cardinality taken over all $k$-tuple domatic partitions of $G$. A $k$-tuple dmatic partition of $G$ of maximum cardinality is called a $d_{\times k}(G)$-partition. The study of this kind of partitions was first begun by Harary and Haynes in \cite{hh2} as a generalization of the well-known domatic partition in graphs (\cite{ch}). When $k=1$, $d_{\times k}(G)$ and ``$k$-tuple domatic partition" are simply written as $d(G)$ and ``domatic partition", respectively. 

This paper is organized as follows. We first present some properties of the Slater number $s\ell_{\times2}(G)$ of graphs $G$. In particular, we observe that $s\ell_{\times2}(G)$ is a lower bound on $\gamma_{\times2}(G)$ for all graphs $G$ with no isolated vertices. We than prove that the problem deciding whether the equality holds for a given graph $G$ is NP-complete (even for $4$-partite graphs) despite that fact that $s\ell_{\times2}(G)$ can be computed in linear-time. We also prove that the problem of computing $\gamma_{\times k}$ for $k\geq2$ is NP-hard for comparability graphs of diameter two. Some bounds on the double domination number in this paper improve the main results in \cite{c} and \cite{hj}. An upper bound on the $k$-tuple domatic number of graphs with characterization of the extremal graphs for the bound is given in this paper. Finally, we give a complete characterization of full graphs (those graphs $G$ for which $d(G)=d_{\times1}(G)=\delta(G)+1$), solving an open problem given in \cite{ch}.

Note that for some other domination parameters, lower bounds similar to the double Slater number can be obtained. The reader can consult \cite{aabd}, \cite{dh} and \cite{gr} for more pieces of information about them.  
 

\section{Preliminary properties of double Slater number}

In this section, we prove some results on $\gamma_{\times2}(G)$ and $s\ell_{\times2}(G)$ and discuss their relationship for general graphs $G$. Let $G$ be a graph of order $n$ with $\delta(G)\geq2$ and degree sequence $d_{1}\geq\cdots\geq d_{n}$. Let $A$ be a $\gamma_{\times2}(G)$-set. We then have
\begin{equation*}\label{E0}
\begin{array}{lcl}
|A|+\sum_{i=1}^{|A|}d_{i}&\geq&|A|+\sum_{v\in A}\deg(v)=|A|+\sum_{v\in A}|N(v)\cap A|+\sum_{v\in A}|N(v)\cap (V(G)\setminus A)|\\
&=&|A|+|A|+2(n-|A|)=2n.
 \end{array}
\end{equation*}
Therefore,
\begin{equation}\label{E-1} 
\gamma_{\times2}(G)\geq s\ell_{\times2}(G)=\mbox{min}\{t\mid t+d_{1}+\cdots+d_{t}\geq2n\}.
\end{equation}

In spite of the fact that the lower bound given in (\ref{E-1}) and its proof are simple, the family of graphs for which the equality holds cannot be characterized in polynomial-time (even if we restrict the problem to some special families of graphs) unless P=NP. We will prove this later in Theorem \ref{Complexity}. 

In the next two propositions, we discuss some properties of the double Slater number. 

\begin{proposition}\label{T-1}
Let $G$ be a graph of order $n$ with minimum degree $\delta\geq2$ and maximum degree $\Delta$. Then, 
$$\lceil\frac{2n}{1+\Delta}\rceil\leq s\ell_{\times2}(G)\leq \lceil\frac{2n}{1+\delta}\rceil.$$
\end{proposition}
\begin{proof}
Let $t=s\ell_{\times2}(G)$. Then, $t+\Delta t\geq t+d_{1}+\cdots+d_{t}\geq2n$. So, $s\ell_{\times2}(G)=t\geq \lceil2n/(1+\Delta)\rceil$. On the other hand,
$$\lceil\frac{2n}{1+\delta}\rceil+d_{1}+\cdots+d_{\lceil\frac{2n}{1+\delta}\rceil}\geq \lceil\frac{2n}{1+\delta}\rceil+\delta \lceil\frac{2n}{1+\delta}\rceil=(1+\delta)\lceil\frac{2n}{1+\delta}\rceil\geq2n.$$
Therefore, $s\ell_{\times2}(G)\leq \lceil2n/(1+\delta)\rceil$.
\end{proof}

As an immediate consequence of Proposition \ref{T-1}, we have $s\ell_{\times2}(G)=\lceil2n/(1+r)\rceil$ for any $r$-regular graph $G$.

\begin{proposition}\label{P1}
Let $G$ be a graph of order $n$ and size $m$ with minimum degree $\delta\geq2$. Then, the following statements hold.\vspace{1.25mm}\\
\emph{(}i\emph{)} $1\leq s\ell_{\times2}(G)-s\ell(G)\leq\lceil n/(\delta+1)\rceil$. These bounds are sharp.\vspace{1.25mm}\\
\emph{(}ii\emph{)} $2\leq s\ell_{\times2}(G)\leq n$. Moreover, $s\ell_{\times2}(G)=2$ if and only if $G$ has at least two vertices of degree $n-1$, and $s\ell_{\times2}(G)=n$ if and only if $m\leq\lfloor(n+\delta)/2\rfloor$.\vspace{1.25mm}\\
\emph{(}iii\emph{)} $s\ell_{\times2}(G)=2n/(1+\Delta)$ if and only if $2n\equiv0$ \emph{(}mod $1+\Delta$\emph{)} and $d_{2n/(1+\Delta)}=\Delta$.
\end{proposition}
\begin{proof}
($i$) Let $t=s\ell_{\times2}(G)$. Then, $t+d_{1}+\cdots+d_{t}\geq2n$. So, $t-1+d_{1}+\cdots+d_{t-1}\geq2n-1-d_{t}\geq n$. Therefore, $s\ell(G)\leq s\ell_{\times2}(G)-1$ by the definition of $s\ell(G)$. We now let $t'=s\ell(G)$. We have
\begin{equation*}
\begin{array}{lcl}
t'+\lceil n/(\delta+1)\rceil&+&d_{1}+\cdots+d_{t'}+d_{t'+1}+\cdots+d_{t'+\lceil n/(\delta+1)\rceil}\\
&\geq& n+\lceil n/(\delta+1)\rceil+\delta\lceil n/(\delta+1)\rceil\geq2n.
\end{array}
\end{equation*}
Thus, $s\ell_{\times2}(G)\leq s\ell(G)+\lceil n/(\delta+1)\rceil$.

The lower bound is sharp for the complete graph $K_{n}$ when $n\geq3$. Moreover, it is easy to see that the upper bound is sharp for the cycle $C_{n}$ in which $n\equiv0$ or $2$ (mod $3$).

($ii$) We only prove the second ``if and only if" part, as the other parts can be verified easily. Let $s\ell_{\times2}(G)=n$. This implies that $d_{1}+\cdots+d_{n-1}\leq n$. So, $2m\leq n+d_{n}=n+\delta$ which results in $m\leq\lfloor(n+\delta)/2\rfloor$. Conversely, let $m\leq\lfloor(n+\delta)/2\rfloor$. Then, 
$$n-1+d_{1}+\cdots+d_{n-1}=n-1+2m-d_{n}=n-1+2m-\delta\leq n-1+2\lfloor(n+\delta)/2\rfloor-\delta<2n.$$
Therefore, $s\ell_{\times2}(G)=n$.

($iii$) We have $s\ell_{\times2}(G)\geq2n/(1+\Delta)$ by Proposition \ref{T-1}. Suppose that the equality holds. Obviously, $2n\equiv0$ (mod $1+\Delta$). On the other hand, 
$$2n\leq2n/(1+\Delta)+d_{1}+\cdots+d_{2n/(1+\Delta)}\leq2n/(1+\Delta)+2n\Delta/(1+\Delta)=2n$$
implies that $d_{2n/(1+\Delta)}=\Delta$. Conversely, we deduce that $2n/(1+\Delta)+d_{1}+\cdots+d_{2n/(1+\Delta)}=2n$ when $2n\equiv0$ (mod $1+\Delta$) and $d_{2n/(1+\Delta)}=\Delta$.
\end{proof}

\begin{rem}
Harary and Haynes \cite{hh} proved that $\gamma_{\times2}(G)\geq2n/(1+\Delta)$ for all graphs $G$ with no isolated vertices. So, the lower bound in Proposition \ref{T-1} gives us an improvement of the result in \cite{hh}. In what follows, we show that the difference between $s\ell_{\times2}(G)$ and $2n/(1+\Delta)$ can be arbitrarily large. In fact, we claim that for any integer $b\geq1$, the difference between these two graph parameters can be as large as $b$. To see this, we begin with the star $K_{1,n-1}$ with the central vertex $u$ in which $n=4b+4$. Let $G$ be obtained from $K_{1,n-1}$ by constructing a path on the vertices in $N_{K_{1,n-1}}(u)$. Note that all vertices in $N_{K_{1,n-1}}(u)$ are of degree three except for two vertices of degree two. Clearly, $2n/(1+\Delta)=2$. We have $d_{1}=n-1$ and $d_{i}=3$ for all $2\leq i\leq4b+2$. We therefore have,
$$b+2+d_{1}+\cdots+d_{b+2}=2n.$$
Thus, $s\ell_{\times2}(G)=b+2$. Therefore $s\ell_{\times2}(G)-2n/(1+\Delta)=b$, as desired.
\end{rem}


\section{Complexity results}

Note that the non-increasing degree sequence of a graph $G$ of order $n$ can be made in linear-time by the counting sort algorithm. Therefore sum of the first $t$ degrees, and hence $s\ell_{\times2}(G)$, can be computed in linear-time. In spite of this fact, the following theorem shows that the problem of determining whether the lower bound (\ref{E-1}) holds with equality for a given graph $G$ is NP-complete even if we restrict our attention to some special families of graphs.

\begin{theorem}\label{Complexity}
The problem of deciding whether $\gamma_{\times2}(G)=s\ell_{\times2}(G)$ for a given graph $G$ is NP-complete even when restricted to 4-partite graphs. 
\end{theorem}
\begin{proof}
We describe a polynomial transformation of the well-known $3$-SAT problem to our problem. Consider an arbitrary instance of the $3$-SAT problem, given by $U=\{u_{1},\cdots,u_{a}\}$ of variables (with the set of complements $U'=\{u_{1}',\cdots,u_{a}'\}$) and a collection $C=\{C_{1},\cdots,C_{b}\}$ of three-variable clauses over $U\cup U'$. Note that the $3$-SAT problem remains NP-complete even if for each $u_{i}\in U$, there are at most five clauses in $C$ that contain either $u_{i}$ or $u_{i}'$ (see \cite{gj}). 

For every variable $u_{i}$, we associate a graph $H_{i}$ of order $5a^{2}$ constructed by a copy of the triangle $K^{i}_{3}$ on vertices $q_{i}$, $q_{i}'$ and $q_{i}''$, by adding $5a^{2}-3$ independent vertices so that each of them is adjacent to all three vertices of $K^{i}_{3}$. Corresponding to each three-variable clause $C_{j}$, we create a vertex $c_{j}$. The construction of the graph $G$ is completed by adding the edges $c_{j}q_{i}''$ and $c_{j}y$, where $y=q_{i}$ or $q_{i}'$ when $u_{i}$ or $u_{i}'$ belongs to $C_{j}$, respectively. Clearly, $G$ is a graph of order $n=5a^{3}+b$ and $\Delta(G)\in\big{\{}\deg(v)\mid v\in W\cup\{q_{1}'',\cdots,q_{a}''\}\big{\}}$, in which $W=\{q_{1},\cdots,q_{a}\}\cup\{q_{1}',\cdots,q_{a}'\}$. It is obvious from the construction that $G$ is $4$-partite (see Figure \ref{Fig1}). Furthermore, $s\ell_{\times2}(G)=\mbox{min}\{t\mid t+d_{1}+\cdots+d_{t}\geq2n\}$ since $\delta(G)\geq2$.

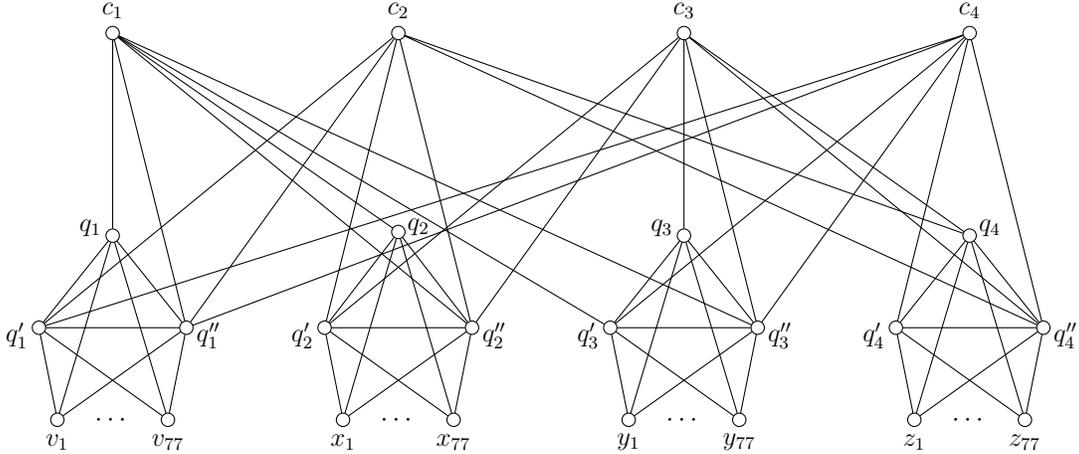
\begin{figure}[ht]
 \centering
\begin{tikzpicture}[scale=.49, transform shape]
\node [draw, shape=circle] (q1) at (0,-0.5) {};
\node [draw, shape=circle] (q1') at (-2,-3) {};
\node [draw, shape=circle] (q1'') at (2,-3) {};
\draw (q1)--(q1')--(q1'')--(q1);
\node [scale=1.7] at (-0.6,-0.3) {\large $q_{1}$};
\node [scale=1.7] at (-2.6,-3.2) {\large $q_{1}'$};
\node [scale=1.7] at (2.6,-3.2) {\large $q_{1}''$};
\node [draw, shape=circle] (w1) at (-1.5,-5.5) {};
\node [draw, shape=circle] (w1') at (1.5,-5.5) {};
\node [scale=1.7] at (0,-5.5) {\large $\cdots$};
\draw (w1)--(q1)--(w1');
\draw (w1)--(q1')--(w1');
\draw (w1)--(q1'')--(w1');
\node [scale=1.7] at (-1.5,-6.1) {\large $v_{1}$};
\node [scale=1.7] at (1.5,-6.1) {\large $v_{77}$};


\node [draw, shape=circle] (q2) at (7.75,-0.4) {};
\node [draw, shape=circle] (q2') at (5.75,-3) {};
\node [draw, shape=circle] (q2'') at (9.75,-3) {};
\draw (q2)--(q2')--(q2'')--(q2);
\node [scale=1.7] at (8.3,-0.25) {\large $q_{2}$};
\node [scale=1.7] at (5.15,-3.2) {\large $q_{2}'$};
\node [scale=1.7] at (10.35,-3.2) {\large $q_{2}''$};
\node [draw, shape=circle] (w2) at (6.25,-5.5) {};
\node [draw, shape=circle] (w2') at (9.25,-5.5) {};
\node [scale=1.7] at (7.75,-5.5) {\large $\cdots$};
\draw (w2)--(q2)--(w2');
\draw (w2)--(q2')--(w2');
\draw (w2)--(q2'')--(w2');
\node [scale=1.7] at (6.25,-6.1) {\large $x_{1}$};
\node [scale=1.7] at (9.25,-6.1) {\large $x_{77}$};


\node [draw, shape=circle] (q3) at (15.5,-0.5) {};
\node [draw, shape=circle] (q3') at (13.5,-3) {};
\node [draw, shape=circle] (q3'') at (17.5,-3) {};
\draw (q3)--(q3')--(q3'')--(q3);
\node [scale=1.7] at (14.9,-0.3) {\large $q_{3}$};
\node [scale=1.7] at (12.9,-3.2) {\large $q_{3}'$};
\node [scale=1.7] at (18.1,-3.2) {\large $q_{3}''$};
\node [draw, shape=circle] (w3) at (14,-5.5) {};
\node [draw, shape=circle] (w3') at (17,-5.5) {};
\node [scale=1.7] at (15.5,-5.5) {\large $\cdots$};
\draw (w3)--(q3)--(w3');
\draw (w3)--(q3')--(w3');
\draw (w3)--(q3'')--(w3');
\node [scale=1.7] at (14,-6.1) {\large $y_{1}$};
\node [scale=1.7] at (17,-6.1) {\large $y_{77}$};


\node [draw, shape=circle] (q4) at (23.25,-0.5) {};
\node [draw, shape=circle] (q4') at (21.25,-3) {};
\node [draw, shape=circle] (q4'') at (25.25,-3) {};
\draw (q4)--(q4')--(q4'')--(q4);
\node [scale=1.7] at (23.8,-0.3) {\large $q_{4}$};
\node [scale=1.7] at (20.65,-3.2) {\large $q_{4}'$};
\node [scale=1.7] at (25.85,-3.2) {\large $q_{4}''$};
\node [draw, shape=circle] (w4) at (21.75,-5.5) {};
\node [draw, shape=circle] (w4') at (24.75,-5.5) {};
\node [scale=1.7] at (23.25,-5.5) {\large $\cdots$};
\draw (w4)--(q4)--(w4');
\draw (w4)--(q4')--(w4');
\draw (w4)--(q4'')--(w4');
\node [scale=1.7] at (21.75,-6.1) {\large $z_{1}$};
\node [scale=1.7] at (24.75,-6.1) {\large $z_{77}$};


\node [draw, shape=circle] (c1) at (0,5) {};
\node [draw, shape=circle] (c2) at (7.75,5) {};
\node [draw, shape=circle] (c3) at (15.5,5) {};
\node [draw, shape=circle] (c4) at (23.25,5) {};

\node [scale=1.7] at (0,5.6) {\large $c_{1}$};
\node [scale=1.7] at (7.75,5.6) {\large $c_{2}$};
\node [scale=1.7] at (15.5,5.6) {\large $c_{3}$};
\node [scale=1.7] at (23.25,5.6) {\large $c_{4}$};

\draw (q1)--(c1)--(q2);
\draw (q1'')--(c1)--(q2'');
\draw (q3')--(c1)--(q3'');

\draw (q4)--(c2)--(q4'');
\draw (q2')--(c2)--(q2'');
\draw (q1')--(c2)--(q1'');

\draw (q3)--(c3)--(q3'');
\draw (q4)--(c3)--(q4'');
\draw (q2')--(c3)--(q2'');

\draw (q4')--(c4)--(q4'');
\draw (q3')--(c4)--(q3'');
\draw (q1')--(c4)--(q1'');

\end{tikzpicture}
\caption{{\small The graph $G$ when $a=b=4$. Here $U=\{u_{1},u_{2},u_{3},u_{4}\}$, $C_1=\{u_{1},u_{2},u_{3}'\}$, $C_2=\{u_{4},u_{2}',u_{1}'\}$, $C_3=\{u_{3},u_{4},u_{2}'\}$ and $C_4=\{u_{4}',u_{3}',u_{1}'\}$. We observe that $(u_{1},u_{2},u_{3},u_{4}):(\mbox{True},\mbox{False},\mbox{True},\mbox{False})$ is a satisfying truth assignment. Note that $\{q_{1},q_{1}'',q_{2}',q_{2}'',q_{3},q_{3}'',q_{4}',q_{4}''\}$ is a $\gamma_{\times2}(G)$-set.}}\label{Fig1}
\end{figure} 

Let $d_{1}\geq\cdots\geq d_{n}$ be the non-increasing degree sequence of $G$. Note that any vertex $c_{j}$ is adjacent to at least five vertices in $W\cup\{q_{1}'',\cdots,q_{a}''\}$, and this minimum adjacency happens if and only if $C_{j}$ contains both $u_{i}$ and $u_{i}'$ for some $1\leq i\leq a$. Therefore, we have
\begin{equation}\label{Ave}
5b\leq\big{|}[\{c_{1},\cdots,c_{b}\},W\cup\{q_{1}'',\cdots,q_{a}''\}]\big{|}=\sum_{v\in W\cup\{q_{1}'',\cdots,q_{a}''\}}\deg_{\{c_{1},\cdots,c_{b}\}}(v).
\end{equation}

We let $\big{\{}\deg_{\{c_{1},\cdots,c_{b}\}}(v)\mid v\in W\cup\{q_{1}'',\cdots,q_{a}''\}\big{\}}=\{d_{1}',\cdots,d_{3a}'\}$ in which $d_{1}'\geq\cdots\geq d_{3a}'$. Therefore, the inequality (\ref{Ave}) can be written as $5b\leq \sum_{i=1}^{3a}d_{i}'$. Hence, $5b/2\leq \sum_{i=1}^{\lceil3a/2\rceil}d_{i}'$. On the other hand, it is clear from the construction that the terms in $\sum_{i=1}^{2a}d_{i}$ correspond to $2a$ vertices in $W\cup\{q_{1}'',\cdots,q_{a}''\}$. By renaming the vertices, we can write $W\cup\{q_{1}'',\cdots,q_{a}''\}=\{w_{1},\cdots,w_{3a}\}$ in such a way that $d_{i}=\deg(w_{i})$ for each $1\leq i\leq3a$. Therefore,
\begin{equation*}
\begin{array}{lcl}
2a+\sum_{i=1}^{2a}d_{i}=2a+\sum_{i=1}^{2a}\deg(w_{i})&=&2a+\sum_{i=1}^{2a}\deg_{V(H_i)}(w_{i})+\sum_{i=1}^{2a}\deg_{\{c_{1},\cdots,c_{b}\}}(w_{i})\vspace{0.5mm}\\
&\geq&2a+2a(5a^{2}-1)+\sum_{i=1}^{\lceil3a/2\rceil}\deg_{\{c_{1},\cdots,c_{b}\}}(w_{i})\vspace{0.5mm}\\
&=&2a+2a(5a^{2}-1)+\sum_{i=1}^{\lceil3a/2\rceil}d_{i}'\\
&\geq&2a+2a(5a^{2}-1)+5b/2\\
&\geq&2n.
\end{array}
\end{equation*}

On the other hand, since every variable in $W$ belongs to at most five clauses in $C$, it follows that $\Delta(G)\leq5a^{2}+4$. So, $2a-1+\sum_{i=1}^{2a-1}d_{i}\leq2a-1+(2a-1)(5a^{2}+4)<2n$. This implies that $s\ell_{\times2}(G)=2a$.

Now let the above-mentioned instance of $3$-SAT be satisfiable. It follows that the subset of those vertices from $W$ corresponding to $a$ variables assigned TRUE along with their associated vertices $q_{i}''$ forms a double dominating set in $G$ of cardinality $2a$. Therefore, $\gamma_{\times2}(G)\leq2a$. This shows that $\gamma_{\times2}(G)=s\ell_{\times2}(G)=2a$. 

Conversely, if $\gamma_{\times2}(G)=s\ell_{\times2}(G)$, then it must happen that $\gamma_{\times2}(G)=2a$. Let $S$ be a $\gamma_{\times2}(G)$-set. Note that each graph $H_{i}$ must have at least two vertices in $S$ so that all vertices in $V(H_{i})$ can be double dominated by $S$. Since $\gamma_{\times2}(G)=|S|=2a$, it follows that $|V(H_{i})\cap S|=2$ for all $1\leq i\leq a$. Taking into account the edges with one end point in $\{c_{1},\cdots,c_{b}\}$, we may assume that all vertices $q_{i}''$ and just one of the vertices $q_{i}$ and $q_{i}'$, for each $1\leq i\leq a$, belong to $S$. We now create a satisfying truth assignment for $C$ by assigning the value TRUE to those variables $u_{i},u_{j}'\in U\cup U'$ for which $q_{i},q_{j}'\in S\cap W$. This completes the proof.
\end{proof}

For various reasons, the small values of $k$ (especially $k\in\{1,2\}$) regarding $\gamma_{\times k}$ attracts more attention from the experts in domination theory rather than the large ones. In particular, this parameter cannot be defined for some important families of graphs like trees when $k>2$; many results for the case $k\in\{1,2\}$ can be generalized to the general case $k$; one may obtain stronger results for the small values of $k$ rather than the large ones. That is why we would rather emphasis on the case $k=2$ in this paper (note that the case $k=1$ leads to the usual domination in graphs). Notwithstanding this, some interesting papers treated this topic from the general point of view. For instance, Liao and Chang \cite{lc} proved that the decision problem $k$-TUPLE DOMINATING SET (associated with $\gamma_{\times k}$) is NP-complete even for split graphs and for bipartite graphs. They also posed the question ``what are the complexities of the $k$-tuple domination for other subclasses of perfect graphs?" 

Here we prove that this decision problem remains NP-complete even for a very restricted subfamily of comparability graphs, that is, comparability graphs of diameter two. To this aim, we first need to recall that the {\em corona product} $G\odot H$ of graphs $G$ (with $V(G)=\{v_{1},\cdots,v_{n}\}$) and $H$ is obtained from the disjoint union of $G$ and $n$ disjoint copies of $H$, say $H_1,\ldots,H_{n}$, such that for all $i\in \{1,\dots,n\}$, the vertex $v_i\in V(G)$ is adjacent to every vertex of $H_i$.

\begin{theorem}\label{cor}
For any integer $k\geq2$, the $k$-TUPLE DOMINATING SET problem is NP-complete even when restricted to comparability graphs of diameter two.
\end{theorem}
\begin{proof}
The problem clearly belongs to NP since checking that a given set is indeed a $k$-tuple dominating set of cardinality at most $j$ can be done in polynomial time.

We set $j=r+k-1$. Let $V(H)=\{h_{1},\cdots,h_{|V(H)|}\}$ and $V(H_{i})=\{h^{i}_{1},\cdots,h^{i}_{|V(H)|}\}$ for each $1\leq i\leq|V(G)|$. Let $S$ be a $\gamma_{\times k}(G\odot H)$-set. We set $S_{i}=S\cap(V(H_{i})\cup\{v_{i}\})$ for each $1\leq i\leq|V(G)|$. Suppose first that $v_{i}\notin S_{i}$. It follows that $S_{i}$ is a $k$-tuple dominating set of $H_{i}$. On the other hand, it turns out that $S_{i}\setminus\{h^{i}_{j}\}$ is a $(k-1)$-tuple dominating set of $H_{i}$, where $h^{i}_{j}$ is any vertex of $S_{i}$. Therefore, $|S_i|\geq \gamma_{\times(k-1)}(H_{i})+1$. Suppose now that $v_{i}\in S_{i}$. In such a situation, it is readily observed that $S_{i}\setminus\{v_{i}\}$ is a $(k-1)$-tuple dominating set of $H_{i}$. Therefore, we have again $|S_i|\geq \gamma_{\times(k-1)}(H_{i})+1$. Consequently, $\gamma_{\times k}(G\odot H)=|S|=\sum_{i=1}^{|V(G)|}|S_i|\geq|V(G)|(\gamma_{\times(k-1)}(H)+1)$. 

Conversely, let $A$ be a $\gamma_{\times(k-1)}(H)$-set. Clearly, $A_{i}=\{h^{i}_{j}\in V(H_{i})\mid h_{j}\in A\}$ is a $\gamma_{\times(k-1)}(H_{i})$-set for every $1\leq i\leq|V(G)|$. By the definition, we can easily check that $A'=\cup_{i=1}^{|V(G)|}A_{i}\cup V(G)$ is a $k$-tuple dominating set in $G\odot H$ of cardinality $|V(G)|(\gamma_{\times(k-1)}(H)+1)$. This results in the exact formula 
\begin{equation}\label{cor-tuple1}
\gamma_{\times k}(G\odot H)=|V(G)|(\gamma_{\times(k-1)}(H)+1)
\end{equation}
for any graphs $G$ and $H$. In particular, we have $\gamma_{\times k}(K_{1}\odot H)=\gamma_{\times(k-1)}(H)+1$ for any graph $H$. We now define the sequence $\{H_{1},H_{2},\cdots\}$ by $H_{1}=H$, $H_{t}=K_{1}\odot H_{t-1}$ for each $t\geq2$. With this in mind and by using the equality (\ref{cor-tuple1}) for $G=K_1$, we get 
\begin{equation*}\label{cor-tuple2}
\gamma_{\times k}(H_{k})=\gamma_{\times(k-1)}(H_{k-1})+1=\cdots=\gamma_{\times1}(H_1)+k-1.
\end{equation*}
Thus, $\gamma_{\times k}(H_{k})=\gamma(H)+k-1$ for every graph $H$. Our reduction is now completed by taking into account that $\gamma_{\times k}(H_{k})\leq j$ if and only if $\gamma(H)\leq r$. 

On the other hand, let $H=H_1$ be a comparability graph. Hence, it has a transitive orientation $D=D_1$. Let $V(H_2)=V(H)\cup\{x\}=V(H_1)\cup\{x\}$. It is easy to check that the orientation $D_{1}$ of $H_1$ along with the set of new arcs $\{(x,v)\mid v\in V(H)=V(H_1)\}$ give us a transitive orientation of $H_2$. Therefore, $H_{2}$ is a comparability graph. Iterating this process, we have that all $H_{1},\cdots,H_{k}$ are comparability graphs. 

Because the DOMINATING SET problem (associated with the usual domination number) is NP-complete for comparability graphs (see \cite{cs} or \cite{Dew}) and since $\mbox{diam}(H_{k})\leq2$, we deduce that the $k$-TUPLE DOMINATING SET is NP-complete for comparability graphs of diameter two. This completes the proof.
\end{proof}


\section{Bounding the double domination number}

Let $H$ be a bipartite graph with partite sets $X$ and $Y$ such that every vertex in $X$ has degree two. We make a matching $M$ whose edges join some vertices in $Y$ which have neighbors in $X$. We also join at least one end-vertex to the vertices in $Y$ which are not saturated by $M$. Let $G$ be the constructed graph and let $\Omega$ be the family of such graphs $G$ (see Figure \ref{fig-omega} in the case of a tree $T$).

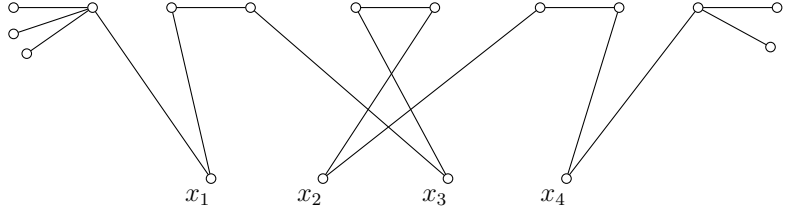
\begin{figure}[ht]
 \centering
\begin{tikzpicture}[scale=.35, transform shape]
\node [draw, shape=circle] (v0) at  (6,14) {};
\node [draw, shape=circle] (v0') at  (3,14) {};
\node [draw, shape=circle] (v0'') at  (5.75,12.5) {};
\node [draw, shape=circle] (v1) at  (0,14) {};
\node [draw, shape=circle] (v1') at  (-3,14) {};
\node [draw, shape=circle] (v2) at  (-7,14) {};
\node [draw, shape=circle] (v2') at  (-10,14) {};
\node [draw, shape=circle] (v3) at  (-14,14) {};
\node [draw, shape=circle] (v3') at  (-17,14) {};
\node [draw, shape=circle] (v4) at  (-20,14) {};
\node [draw, shape=circle] (v4') at  (-23,14) {};

\node [draw, shape=circle] (x4) at  (-2,7.5) {};
\node [draw, shape=circle] (x3) at  (-6.5,7.5) {};
\node [draw, shape=circle] (x2) at  (-11.25,7.5) {};
\node [draw, shape=circle] (x1) at  (-15.5,7.5) {};

\node [draw, shape=circle] (w1) at  (-22.5,12.25) {};
\node [draw, shape=circle] (w2) at  (-23,13) {};

\draw (v0)--(v0');
\draw (v0')--(v0'');
\draw (v1)--(v1');
\draw (v2)--(v2');
\draw (v3)--(v3');
\draw (v4)--(v4');
\draw (v3')--(x1)--(v4);
\draw (v2)--(x2)--(v1');
\draw (v1)--(x4)--(v0');
\draw (v3)--(x3)--(v2');
\draw (w1)--(v4)--(w2);

\node [scale=2.4] at (-2.5,6.8) {\large $x_{4}$};
\node [scale=2.4] at (-7,6.8) {\large $x_{3}$};
\node [scale=2.4] at (-11.75,6.8) {\large $x_{2}$};
\node [scale=2.4] at (-16,6.8) {\large $x_{1}$};

\end{tikzpicture}
  \caption{{\small A tree $T\in \Omega$ with $X=\{x_{1},x_{2},x_{3},x_{4}\}$ and $|M|=3$.}}\label{fig-omega}
\end{figure}

\begin{theorem}\label{general}
Let $G$ be a graph of order $n$ and size $m$ with no isolated vertices. Let $e$ and $p$ be the number of end-vertices and penultimate vertices in $G$, respectively. Then,
$$\gamma_{\times2}(G)\geq\frac{4n-2m+e-p}{3}.$$
Furthermore, the equality holds if and only if $G\in \Omega$.
\end{theorem}
\begin{proof}
Let $A$ be a $\gamma_{\times2}(G)$-set. The definition of a double dominating set implies that every vertex not in $A$ is adjacent to at least two vertices in $A$, and every vertex in $A$ is adjacent to at least one vertex in $A$. Moreover, all end-vertices and penultimate vertices belong to $A$, necessarily. Therefore,
\begin{equation}\label{Ine1}
\begin{array}{lcl}
m=|E(G[A])|+|[A,V(G)\setminus A]|+|E(G[V(G)\setminus A])|\geq e+(|A|-e-p)/2+2(n-|A|).
\end{array}
\end{equation}
Thus, $\gamma_{\times2}(G)=|A|\geq(4n-2m+e-p)/3$.

Suppose first that $G\in \Omega$. Then $V(G)$ is the union of vertices of $|M|$ copies of $P_{2}$, $|Y|-2|M|$ stars and $|X|$ independent vertices. Hence, $n=2|M|+|X|+e+p$ and $m=|M|+2|X|+e$. It is then easily checked that the set $S$ containing the saturated vertices by $M$ along with the vertices of the $|Y|-2|M|$ stars is a double dominating set in $G$. Therefore, $\gamma_{\times2}(G)\leq|S|=2|M|+e+p=(4n-2m+e-p)/3$. This ends up with equality in the lower bound.

Conversely, let the equality hold. So the inequality in (\ref{Ine1}) holds with equality, necessarily. In particular, $|E(G[A])|=e+(|A|-e-p)/2$ implies that the subgraph induced by $A$ is a disjoint union of $P_{2}$-copies and stars $T_{i}$ with partite sets $X_{i}$ and $Y_{i}$ with $|X_{i}|=1$, for which $\cup_{i}Y_{i}$ and $\cup_{i}X_{i}$ are the sets of end-vertices and penultimate vertices of $G$, respectively. Moreover, $|[A,V(G)\setminus A]|=2(n-|A|)$ implies that every vertex in $V(G)\setminus A$ has precisely two neighbors in $A$. Finally, the vertices not in $A$ are independent as $E(G[V(G)\setminus A])$ is empty. It is now easily observed that $V(G)\setminus A$, $A$, the set of edges of $P_2$-copies, $\cup_{i}X_{i}$ and $\cup_{i}Y_{i}$ correspond to $X$, $Y$, $M$, $M$-unsaturated vertices of $Y$ and the end-vertices described in the process of defining $\Omega$, respectively. Thus, $G\in \Omega$.
\end{proof}

Note that Theorem \ref{T4} implies the lower bound given in Theorem \ref{general}. In spite of this, we proved it by a different method so as to give the characterization of graphs attaining the lower bound.

As an immediate consequence of Theorem \ref{general} for trees we have the following result of Chellali in $2006$.

\begin{theorem}\emph{(\cite{c})}\label{T2}
If $T$ is a nontrivial tree of order $n$ with $\ell$ leaves and $s$ support vertices, then $\gamma_{\times2}(T)\geq(2n+\ell-s+2)/3$.
\end{theorem}
Note that the whole of paper \cite{c} is devoted to the bound given in Theorem \ref{T2} and a constructive characterization of all trees attaining it.

In order to characterize all trees for which the equality holds in the bound given in Theorem \ref{T2}, it suffices to restrict the family $\Omega$ to trees. Let us denote the resulting family by $\Omega'$. In what follows, we describe a typical member of $\Omega'$. We begin with $a$ copies of $P_{2}$ and $s$ copies of stars $T_{i}$ with partite sets $X_{i}$ and $Y_{i}$ such that $|X_{i}|=1$. We then add $r=a+s-1$ new vertices and join each of them to precisely two vertices in the $P_{2}$-copies and $X_{i}$'s such that all vertices of $P_{2}$-copies are incident with at least one of them. Note that $m=a+\sum_{i}|Y_{i}|+2r=n-1$ guarantees that the resulting graph is a tree \big{(}see Figure \ref{fig-omega} for $(a,s,|Y_{1}|,|Y_{2}|)=(3,2,3,2)$\big{)}.

Hajian and Jafari Rad in 2019 generalized Theorem \ref{T2} to connected graphs as follows. Note that they used the words ``leaf" and ``support vertex" instead of ``end-vertex" and ``penultimate vertex", respectively.

\begin{theorem}\emph{(\cite{hj})}\label{T3}
If $G$ is a connected graph of order $n\geq2$ with $k\geq0$ cycles, $e$ end-vertices and $p$ penultimate vertices, then $\gamma_{\times2}(G)\geq(2n+e-p+2)/3-2k/3$.
\end{theorem}
They also characterized the family of all graphs achieving the equality in the above bound by extending the characterization given in \cite{c}. 

Let $G$ be a graph of size $m$ given in Theorem \ref{T3} and let $T$ be a spanning tree of it. Clearly, all end-vertices and penultimate vertices of $G$ belong to $V(T)$. Let $k'$ be the number of edges of $G$ which are not in $T$. So, $m=n-1+k'$. On the other hand, it is well-known that for such an edge $xy$, $T+xy$ contains a unique cycle $C_{xy}$ containing $xy$. Moreover, for any two such edges $xy$ and $x'y'$, $C_{xy}=C_{x'y'}$ implies that $xy=x'y'$. This shows that $xy\longrightarrow C_{xy}$ is a one-to-one function. Therefore, $k'\leq k$. We then have 
$$\frac{4n-2m+e-p}{3}=\frac{2n+e-p+2}{3}-\frac{2k'}{3}\geq \frac{2n+e-p+2}{3}-\frac{2k}{3}.$$
In fact, Theorem \ref{general} is an improvement of Theorem \ref{T3}.

\begin{theorem}\label{T4}
For any graph $G$ of order $n$ and size $m$ with $e$ end-vertices, $p$ penultimate vertices and $\delta(G)\geq1$,
$$\gamma_{\times2}(G)\geq s\ell_{\times2}(G)\geq(4n-2m+e-p)/3.$$
Furthermore, $s\ell_{\times2}(T)=(4n-2m+e-p)/3$ if and only if $n+m+e-p\equiv0$ \emph{(}mod $3$\emph{)} and one of the following conditions holds:\vspace{0.5mm}\\ 
\emph{(}$i$\emph{)} If $\delta(G)\geq2$, then $d_{q}=2$, where $q=(4n-2m+3)/3$.\vspace{0.5mm}\\ 
\emph{(}$ii$\emph{)} If $\delta(G)=1$, then \big{(}$n=2m-e+p$\big{)} or \big{(}$n<2m-e+p$ and $d_{q}=2$, where $q=(4n-2m-2e-p+3)/3$\big{)}.  
\end{theorem}
\begin{proof}
Let $A$ be a $\gamma_{\times2}(G)$-set. Let $L$ and $P=\{u_{1},\cdots,u_{p}\}$ be the sets of end-vertices and penultimate vertices of $G$, respectively. It is immediate from the definition that $P\cup L\subseteq A$. Moreover, each vertex in $A$ has at least one neighbor in $A$, and every vertex in $V(G)\setminus A$ has at least two neighbors in $A$. Therefore,
\begin{equation}\label{E1}
\begin{array}{lcl}
|A|+\sum_{i=1}^{|A|-e}d_{i}&\geq&|A|+\sum_{v\in A\setminus L}\deg(v)\\
&=&|A|+\sum_{v\in A\setminus L}|N(v)\cap A|+\sum_{v\in A\setminus L}|N(v)\cap (V(G)\setminus A)|\\
&=&|A|+\sum_{i=1}^{p}|N(u_{i})\cap A|+\sum_{v\in A\setminus(P\cup L)}|N(v)\cap A|\\
&+&\sum_{v\in A\setminus L}|N(v)\cap (V(G)\setminus A)|\\
&\geq&|A|+e+(|A|-e-p)+2(n-|A|)\\
&=&2n-p. 
\end{array}
\end{equation}
Thus, $\gamma_{\times2}(G)=|A|\geq s\ell_{\times2}(G)$.

Now let $t=s\ell_{\times2}(G)$. If $t=n$, then $\gamma_{\times2}(G)=s\ell_{\times2}(G)$ and so the lower bound follows from Theorem \ref{general}. Therefore, we may assume that $t<n$. By the definition, we have
$$2n-p\leq t+d_{1}+\cdots+d_{t-e}=t+2m-(\sum_{i=t-e+1}^{n-e}d_{i}+\sum_{i=n-e+1}^{n}d_{i})\leq t+2m-\big{(}e+2(n-t)\big{)}.$$
Therefore, $s\ell_{\times2}(T)=t\geq(4n-2m+e-p)/3$.

Suppose now that $s\ell_{\times2}(G)=(4n-2m+e-p)/3$. Clearly, $n+m+e-p\equiv0$ (mod $3$). We distinguish two cases depending on the minimum degree of $G$.\vspace{1mm}
 
\textit{Case 1.} $\delta(G)\geq2$. Then, $s\ell_{\times2}(G)=(4n-2m)/3$. If $s\ell_{\times2}(G)=n$, then $n=2m$ that contradicts the fact that $\delta(G)\geq2$. Therefore, $s\ell_{\times2}(G)<n$. Suppose to the contrary that $d_{q}>2$, where $q=(4n-2m+3)/3$. Let $t=s\ell_{\times2}(G)$. We then have
$$2n\leq t+d_{1}+\cdots+d_{t}=t+2m-\sum_{i=t+1}^{n}d_{i}<t+2m-2(n-t).$$
So, $t=s\ell_{\times2}(G)>(4n-2m)/3$. This is a contradiction. Therefore, $d_{q}=2$.\vspace{1mm}

\textit{Case 2.} $\delta(G)=1$. Suppose first that $s\ell_{\times2}(G)=n$. Therefore, $n=2m-e+p$ by the equality in the lower bound. We now assume that $t=s\ell_{\times2}(G)<n$. This implies that $n<2m-e+p$ and that $q=t-e+1<n-e+1$. Therefore, $d_{q}\geq2$. If $d_{q}>2$, then
$$2n-p\leq t+d_{1}+\cdots+d_{t-e}=t+2m-(\sum_{i=t-e+1}^{n-e}d_{i}+\sum_{i=n-e+1}^{n}d_{i})<t+2m-2(n-t)-e$$
implies that $t>(4n-2m+e-p)/3$, which is a contradiction. Therefore, $d_{q}=2$.

Conversely, suppose that $n+m+e-p\equiv0$ (mod $3$) and one of the conditions ($i$) and ($ii$) holds. Again, we consider two cases depending on $\delta(G)$.\vspace{1mm}

\textit{Case 3.} $\delta(G)\geq2$. It is readily seen that $q-1=(4n-2m)/3<n$. We have 
$$q-1+d_{1}+\cdots+d_{q-1}=q-1+2m-\sum_{i=q}^{n}d_{i}=q-1+2m-2(n-q+1)=2n.$$
Therefore, $s\ell_{\times2}(G)\leq q-1$ and so $s\ell_{\times2}(G)=(4n-2m)/3$.\vspace{1mm}

\textit{Case 4.} $\delta(G)=1$. If $n=2m-e+p$, then $(4n-2m+e-p)/3=n=s\ell_{\times2}(G)$ and we are done. So, we may assume that $n<2m-e+p$. This implies that $q\leq n-e$. We now have
\begin{equation*}
\begin{array}{lcl}
q+e-1+d_{1}+\cdots+d_{q-1}&=&q+e-1+2m-(\sum_{i=q}^{n-e}d_{i}+\sum_{i=n-e+1}^{n}d_{i})\\
&=&q+e-1+2m-2(n-e-q+1)-e=2n-p.
\end{array}
\end{equation*}
Therefore, $s\ell_{\times2}(G)\leq q+e-1=(4n-2m+e-p)/3$. This completes the proof.
\end{proof}

\begin{rem}
Note that the gap between $s\ell_{\times2}(T)$ and $(2n+\ell-s+2)/3$ can be arbitrarily large in the case of nontrivial trees $T$ of order $n$ with $\ell$ leaves and $s$ support vertices. In fact, we claim that for any integer $b\geq1$, there exists a tree $T$ for which $s\ell_{\times2}(T)-(2n+\ell-s+2)/3=b$. To see this, let $T$ be obtained from the path $P_{6b+4}$ by joining a leaf to each vertex of $P_{6b+4}$. Then $n=12b+8$ and $\ell=s=6b+4$. Then, $(2n+\ell-s+2)/3=8b+6$. On the other hand, 
$$9b+6+d_{1}+\cdots+d_{9b+6-\ell}=9b+6+d_{1}+\cdots+d_{3b+2}=9b+6+3(3b+2)=18b+12=2n-s$$
shows that $s\ell_{\times2}(T)=9b+6=(2n+\ell-s+2)/3+b$, as desired.
\end{rem}


\section{On $k$-tuple domatic partitioning of a graph}

In this section, we give an upper bound on the $k$-tuple domatic number of graphs in terms of the order, size and $k$-tuple domination number. In order to characterize the extremal graphs which attain the bound, we introduce the family $\Psi$ of graphs as follows. Let $H_{1},\cdots,H_{r}$ be $(k-1)$-regular graphs of the same order. Let $G$ be obtained from $H_{1},\cdots,H_{r}$ by joining each vertex of $H_{i}$ to precisely $k$ vertices of $H_{j}$ for all $1\leq i\neq j\leq r$. Now let $\Psi$ be the family of all such graphs $G$.
 
\begin{theorem}\label{Domatic}
For any graph $G$ of order $n$ and size $m$ with $\delta(G)\geq k-1$,
$$d_{\times k}(G)\leq \frac{1}{2}+\sqrt{\frac{1}{4}+\frac{2m-(k-1)n}{k\gamma_{\times k}(G)}}.$$
Moreover, the equality holds if and only if $G\in \Psi$.
\end{theorem}
\begin{proof}
For the sake of convenience, we write $d_{\times k}=d_{\times k}(G)$ and $\gamma_{\times k}=\gamma_{\times k}(G)$. Let $\mathbb{S}=\{S_{1},\cdots,S_{d_{\times k}}\}$ be a $d_{\times k}$-partition (that is, a $k$-tuple domatic partition of $G$ of cardinality $d_{\times k}$). Without loss of generality, we may assume that $|S_{1}|\leq\cdots\leq|S_{d_{\times k}}|$. By the definition, every vertex in $S_i$ has at least $k-1$ neighbors in $S_i$ as well as $k$ neighbors in $S_j$ for each $1\leq i\neq j\leq d_{\times k}$. Thus, 
\begin{equation}\label{Inequ}
\begin{array}{lcl}
m&=&\sum_{1\leq i\leq d_{\times k}}|[S_{i},S_{i}]|+\sum_{1\leq i<j\leq d_{\times k}}|[S_{i},S_{j}]|\vspace{1mm}\\
&\geq& \sum_{1\leq i\leq d_{\times k}}(k-1)|S_{i}|/2+\sum_{1\leq i<j\leq d_{\times k}}|[S_{i},S_{j}]|\vspace{1mm}\\
&\geq& \dfrac{(k-1)n}{2}+\sum_{i=1}^{d_{\times k}}k|S_{i}|(d_{\times k}-i)\vspace{1mm}\\
&\geq& \dfrac{(k-1)n}{2}+k|S_1|\sum_{i=1}^{d_{\times k}}(d_{\times k}-i)\vspace{1mm}\\
&\geq& \dfrac{(k-1)n}{2}+\big{(}\dfrac{d_{\times k}(d_{\times k}-1)}{2}\big{)}k\gamma_{\times k}.
\end{array}
\end{equation}
This leads to $-k\gamma_{\times k}d_{\times k}^{2}+k\gamma_{\times k}d_{\times k}+2m-(k-1)n\geq0$. Solving this inequality for $d_{\times k}$, we get $d_{\times k}\leq\big{(}1+\sqrt{1+4(2m-(k-1)n)/k\gamma_{\times k}}\big{)}\big{/}2$.

Suppose now that $G\in \Psi$. It is clear that $\{H_{1},\cdots,H_{r}\}$ is a $k$-tuple domatic partition of $G$. In particular, this shows that $\gamma_{\times k}\leq|H_{1}|$. Let $D$ be a $\gamma_{\times k}$-set. We set 
$$A=\{(d,h)\mid d\in D, h\in H_{1}\ \mbox{and}\ h\in N[d]\}.$$ 
Notice that every vertex $h\in H_1$ is $k$-tuple dominated by $D$. This shows that $|A|\geq k|H_1|$. On the other hand, $|N[d]\cap H_{1}|\leq k$ for every $d\in D$. Therefore, $|A|\leq k|D|$. Together these two inequalities imply that $|H_1|\leq|D|=\gamma_{\times k}$, and hence $\gamma_{\times k}=|H_{1}|$. Therefore, $\{H_{1},\cdots,H_{r}\}$ is $k$-tuple domatic partition of $G$ into $\gamma_{\times k}$-sets. In particular, $n=r\gamma_{\times k}$ and $r=d_{\times k}$. Furthermore, $m=(kr-1)n/2$ because $G$ is a $(kr-1)$-regular graph. It is now easy to compute that $\big{(}1+\sqrt{1+4(2m-(k-1)n)/k\gamma_{\times k}}\big{)}\big{/}2=r=d_{\times k}$. So, we have the equality in the upper bound. 

Conversely, let the upper bound hold with equality. Therefore, all inequalities in (\ref{Inequ}) hold with equality, necessarily. In particular, the first two resulting equalities imply that\vspace{0.5mm} 

$(i)$ each vertex in $S_{i}$ has precisely $(k-1)$ neighbors in $S_{i}$, and\vspace{0.5mm} 

$(ii)$ each vertex in $S_{i}$ ($1\leq i\leq d_{\times k}-1$) has precisely $k$ neighbors in $S_{j}$ with $j>i$.\vspace{0.5mm}\\
Therefore, $G[S_{i}]$ is a $(k-1)$-regular graph for each $1\leq i\leq d_{\times k}$, and $|[S_{i},S_{j}]|=k|S_i|$ for all $1\leq i<j\leq d_{\times k}$. While the last two resulting equalities show that $|S_{1}|=\cdots=|S_{d_{\times k}}|=\gamma_{\times k}$. With this in mind, by the equality $|[S_{i},S_{j}]|=k|S_i|$ for all $1\leq i<j\leq d_{\times k}$ along with the fact that each $S_{i}$ is a $k$-tuple dominating set in $G$, we deduce that each vertex in $S_{i}$ has precisely $k$ neighbors in each $S_{j}$ for all $1\leq i\neq j\leq d_{\times k}$. It is now easy to see that $d_{\times k}$, $G[S_1],\cdots,G[S_{d_{\times k}}]$ and $k$ have the same role as $r$, $H_{1},\cdots,H_{r}$ and $k$ have in the process of introducing of $\Psi$, respectively. Therefore, $G\in \Psi$. This completes the proof. 
\end{proof}

Notice that for the special case when $k=1$, we have the usual domination number $\gamma(G)$ and domatic number $d(G)$. In such a situation, the upper bound in Theorem \ref{Domatic} and its associated characterization become much simpler.

\begin{corollary}
Let $G$ be a graph of size $m$. Then
$$d(G)\leq \frac{1}{2}+\sqrt{\frac{1}{4}+\frac{2m}{\gamma(G)}},$$
with equality if and only if $G$ is an $r$-partite graph with partite sets $H_{1},\cdots,H_{r}$ in which the edge set of $G[H_{i}\cup H_{j}]$ is a perfect matching for each $1\leq i\neq j\leq r$. 
\end{corollary}

It is well-known that $\delta(G)+1$ is an upper bound on the domatic number of any graph $G$, and that a graph $G$ is called ({\em domatically}) {\em full} if $d(G)=\delta(G)+1$ (see \cite{ch}). Cockayne and Hedetniemi in \cite{ch} posed the open problem of characterizing all full graphs. In what follows, we give a constructive characterization of all such graphs. To do so, we introduced the family $\Theta$ as follows. Let $H_{1},\cdots,H_{r}$ be any graphs such that at least one of them, say $H_{i}$, has an isolated vertex $v$. Let $G$ be obtained from the disjoint union $H_{1}+\cdots+H_{r}$ by\vspace{0.5mm}\\
$(i)$ joining $v$ to precisely one vertex in each $H_{j}$ with $j\neq i$, and\vspace{0.5mm}\\
$(ii)$ adding some edges among the vertices in $\cup_{t=1}^{r}V(H_{t})\setminus\{v\}$ such that every vertex in $H_{t}$ has at least one neighbor in $H_{t'}$ for each $1\leq t\neq t'\leq r$.\vspace{0.5mm}\\
Let $\Theta$ be the family of all such graphs $G$.

\begin{theorem}\label{full}
A graph $G$ is full if and only if $G\in \Theta$.
\end{theorem}
\begin{proof}
Suppose first that $G\in \Theta$. Clearly, $\deg(v)=\delta(G)=r-1$. By the construction, we observe that $V(H_{t})$ is a dominating set in $G$ for each $1\leq t\leq r$. So, $\mathbb{H}=\{V(H_{1}),\cdots,V(H_{r})\}$ is a domatic partition of $G$. On the other hand, $|\mathbb{H}|=r=\delta(G)+1$ implies that $\mathbb{H}$ is a $d(G)$-set, and so $\delta(G)+1=r=d(G)$.

Suppose now that $d(G)=\delta(G)+1$. Let $\mathbb{Q}=\{Q_{1},\cdots,Q_{d(G)}\}$ be a $d(G)$-partition. Let $v\in Q_{i}$ be a vertex of minimum degree in $G$. Since every set in $\mathbb{Q}$ is a dominating set, $v$ is adjacent to at least one vertex in $Q_{j}$ for each $1\leq i\neq j\leq d(G)$. Therefore, $\deg(v)\geq d(G)-1$. With this in mind, the equality $d(G)=\delta(G)+1$ implies that $v$ is an isolated vertex of $G[Q_{i}]$ having precisely one neighbor in $Q_{j}$ for each $1\leq i\neq j\leq d(G)$. Furthermore, since $\mathbb{Q}$ is a domatic partition of $G$, every vertex different from $v$ in $Q_{t}$ has at least one neighbor in $Q_{t'}$ for each $1\leq t\neq t'\leq d(G)$. It is now easily seen that $d(G)$ and $G[Q_{1}],\cdots,G[Q_{d(G)}]$ correspond to $r$ and $H_{1},\cdots,H_{r}$ in the description of the family $\Theta$. Thus, $G\in \Theta$.   
\end{proof}

The characterization given in Theorem \ref{full} would be much simpler in the case of regular graphs. In fact, if we restrict the family $\Theta$ to regular graphs, the following result will be immediate from Theorem \ref{full}. This result, by a different expression, was proved by Zelinka in \cite{Zelinka}. 

\begin{corollary}
Let $G$ be an $r$-regular graph. Then, $G$ is full if and only if it is an $(r+1)$-partite graph with partite sets $H_{1},\cdots,H_{r+1}$ in which the edges of $G[H_{i}\cup H_{j}]$ form a perfect matching for each $1\leq i\neq j\leq r+1$.
\end{corollary}

\begin{flushleft}
\textbf{{\large Acknowledgments}}\vspace{0.5mm}\\
This work was jointly supported by the Iran National Science Foundation (INSF) and Alzarha University with grant number 99022321.
\end{flushleft}



\begin{thebibliography}{}

\bibitem {aabd} D. Amos, J. Asplund, B. Brimkov and R. Davila, {\em The Slater and sub-$k$-domination number of a graph with applications to domination and $k$-domination}, Discuss. Math. Graph Theory, {\bf 40} (2020), 209--225.
\bibitem {bchh} M. Blidia, M. Chellali, T.W. Haynes and M.A. Henning, {\em Independent and double domination in trees}, Utilitas Math. {\bf 70} (2006), 159--173.
\bibitem {c} M. Chellali, {\em A note on the double domination number in trees}, AKCE J. Graphs. Combin. {\bf 3} (2006), 147--150.
\bibitem {ch} E.J. Cockayne and S.T. Hedetniemi, {\em Towards a theory of domination in graphs}, Networks {\bf 7} (1977), 247--261.
\bibitem {cs} D.G. Corneil and L.K. Stewart, {\em Dominating sets in perfect graphs}, Discrete Math. {\bf 86} (1990), 145--164.
\bibitem {dhh} W.J. Desormeaux, T.W. Haynes and M.A. Henning, {\em Improved bounds on the domination number of a tree}, Discrete Appl. Math. {\bf 177} (2014), 88--94.
\bibitem {dh} W.J. Desormeaux and M.A. Henning, {\em A new lower bound on the total domination number of a tree}, Ars Combin. {\bf 138} (2018), 305--322.
\bibitem {Dew} A.K. Dewdney, Fast Turing reductions between problems in NP 4, Report 71, University of Western Ontario, 1981.
\bibitem {gj} M.R. Garey and D.S. Johnson, Computers and intractability: A guide to the theory of NP-completeness, W.H. Freeman $\&$ Co., New York, USA, 1979.
\bibitem {ghr} M. Gentner, M.A. Henning and D. Rautenbach, {\em Smallest domination number and largest independence number of graphs and forests with given degree sequence}, J Graph Theory, {\bf 88} (2018), 131--145.
\bibitem {gr} M. Gentner and D. Rautenbach, {\em Some comments on the Slater number}, Discrete Math. {\bf 340} (2017), 1497--1502.
\bibitem {hj} M. Hajian and N. Jafari Rad, {\em A new lower bound on the double domination number of a graph}, Discrete Appl. Math. {\bf 254} (2019), 280--282.
\bibitem {hh} F. Harary and T.W. Haynes, {\em Double Domination in Graphs}, Ars Combin. {\bf 55} (2000), 201--213.
\bibitem {hh2} F. Harary and T.W. Haynes, {\em The $k$-tuple domatic number of a graph}, Math. Slovaca, {\bf 48} (1998), 161--166.
\bibitem {hhh} T.W. Haynes, S.T. Hedetniemi and M.A. Henning (editors), Topics in Domination in Graphs, Switzerland:  Springer International Publishing, 2020.
\bibitem {hhs} T.W. Haynes, S.T. Hedetniemi and P.J. Slater, Fundamentals of Domination in Graphs, New York, Marcel Dekker, 1998.
\bibitem {lc} C.S. Liao and G.J. Chang, {\em $k$-tuple domination in graphs}, Inform. Process. Lett. {\bf 87} (2003), 45--50.
\bibitem {s} P.J. Slater, {\em Locating dominating sets and locating-dominating sets}, in: Graph Theory, Combinatorics, and Applications, in: Proc. 7th Quadrennial Int. Conf. Theory Applic. Graphs, {\bf 2} (1995), 1073--1079.
\bibitem {West} D.B. West, Introduction to Graph Theory (Second Edition), Prentice Hall, USA, 2001.
\bibitem {Zelinka} B. Zelinka, {\em Domatically ciritical graphs}, Czechoslovak Math. J. {\bf 30} (1980), 486--89.

\end{thebibliography}
\end{document}